\documentclass[11pt,reqno]{amsart}
\usepackage[T1]{fontenc}
\usepackage{lmodern}
\usepackage{amsmath,amssymb,amsthm,mathtools}
\usepackage[margin=1in]{geometry}
\usepackage[expansion=false]{microtype}
\usepackage{enumitem}
\usepackage{needspace}
\usepackage{tikz}
\usetikzlibrary{arrows,positioning,calc}
\tikzset{property/.style={draw,rounded corners=2pt,align=center,
  inner sep=5pt,font=\small},imp/.style={->,>=latex,semithick},
  equiv/.style={<->,>=latex,semithick}}
\newcommand{\step}[1]{\par\smallskip\noindent\emph{#1}\enspace}
\usepackage[colorlinks=true,linkcolor=black,citecolor=black,urlcolor=black]{hyperref}
\hypersetup{pdfauthor={Hwankoo Kim},pdftitle={Gorenstein Prufer domains with non-Prufer integral closure},
  pdfsubject={Divisoriality, ultraproducts, and divided prime quotients}}
\numberwithin{equation}{section}
\newtheorem{theorem}{Theorem}[section]
\newtheorem{proposition}[theorem]{Proposition}
\newtheorem{corollary}[theorem]{Corollary}
\newtheorem{lemma}[theorem]{Lemma}
\theoremstyle{definition}
\newtheorem{definition}[theorem]{Definition}
\newtheorem{example}[theorem]{Example}
\theoremstyle{remark}

\DeclareMathOperator{\Frac}{Frac}
\DeclareMathOperator{\Hom}{Hom}

\DeclareMathOperator{\Rel}{Rel}
\DeclareMathOperator{\Gfd}{Gfd}
\DeclareMathOperator{\Tor}{Tor}

\DeclareMathOperator{\im}{Im}
\DeclareMathOperator{\fd}{fd}

\newcommand{\wGgldim}{\operatorname{wGgldim}}
\newcommand{\wgd}{\operatorname{w.gl.dim}}

\newcommand{\U}{\mathcal U}
\newcommand{\m}{\mathfrak m}
\setlist[enumerate,1]{label=\textup{(\roman*)},leftmargin=2.3em}
\allowdisplaybreaks[2]

\newtheorem*{maintheorem}{Theorem A}
\title[Gorenstein Pr\"ufer domains and integral closure]{Gorenstein Pr\"ufer domains with\protect\\
non-Pr\"ufer integral closure}
\author[Kim]{Hwankoo Kim}
\address{Department of Computer Engineering, Hoseo University,
Asan 31499, Republic of Korea}
\email{hkkim@hoseo.edu}
\date{}
\begin{document}
\raggedbottom
\begin{abstract}
There exists a one-dimensional local Gorenstein Pr\"ufer domain whose
integral closure is not Pr\"ufer. We prove that the recently constructed
coherent counterexample to the normalization problem has this stronger
Gorenstein property. The argument establishes permanence of divisoriality
of finitely generated ideals under ultraproducts and quotients by divided
primes. It also supplies uniform matrix factorizations: in an ultraproduct
of plane-curve domains, every nonzero ideal generated by $r$ elements is
a cycle of a two-periodic totally acyclic complex with free terms of
rank $r$. The divided prime defining the one-dimensional example is flat,
so these complexes descend through the quotient and localization.
A specified two-generated ideal remains noninvertible in every integral
overring. Over the normalization, its associated Laurent overring has
nonzero first Tor with every injective cogenerator, although both rings
are Gorenstein flat over the original domain. Examples distinguish this
failure from a nonnormal Gorenstein curve with Pr\"ufer normalization
and show that the periodicity conclusion is stronger than the
Gorenstein Pr\"ufer property alone.
\end{abstract}
\maketitle

\section{Introduction}

Normalization and homological finiteness interact especially strongly in
dimension one. The normalization of a one-dimensional Noetherian domain
is Dedekind by the Krull--Akizuki theorem
\cite[Tag~09IG]{Stacks}. Coherence weakens Noetherianity by controlling
relations among finitely many generators rather than requiring every
ideal to be finitely generated. The classical normalization problem
therefore asks how much ideal-theoretic regularity survives under this
weaker finiteness condition.

Gorenstein Pr\"ufer domains impose an additional homological condition:
all torsion-free modules are Gorenstein flat
\cite[Corollary~2.4]{HLZ}. Equivalently, they are coherent domains in
which every nonzero finitely generated ideal is divisorial
\cite[Theorem~4.2]{QW}. They need not be Pr\"ufer themselves; the cusp
$k[[t^2,t^3]]$ is a familiar example. The separate question in
\cite[Section~9, Problem~13]{GKL} is whether the integral closure of a
Gorenstein Pr\"ufer domain must be Pr\"ufer. It concerns the structure
of the larger ring, not merely its Gorenstein flatness over the
original domain.

Tran, Vo, and Nguyen \cite{TVN} construct a one-dimensional coherent
local domain with non-Pr\"ufer normalization. Their construction begins
with an ultraproduct of plane-curve domains, then takes a divided-prime
quotient and a localization. We prove that the same domain is Gorenstein
Pr\"ufer. We further give periodic complete resolutions for all of its
finitely generated ideals and a specified homological obstruction over
its normalization.

The following statement summarizes the conclusions proved in
Theorems~\ref{thm:main} and~\ref{thm:specified-overring}.

\begin{maintheorem}
There exist a one-dimensional local Gorenstein Pr\"ufer domain $D$ and
nonzero elements $a,b\in D$ with the following properties.
\begin{enumerate}
\item The integral closure $B$ of $D$ in $\Frac D$ is not Pr\"ufer.
In fact, $(a,b)T$ is not invertible for any integral overring $T$ of $D$.
\item For every integer $r\ge1$, each nonzero ideal of $D$ generated
by $r$ elements is a cycle of a two-periodic totally acyclic complex
whose terms are $D^r$.
\item Put $z=b/a$ and $S=B[z,z^{-1}]$. For every injective cogenerator
$E_0$ of $B$-modules,
\[
 \Gfd_DB=\Gfd_DS=0,
 \qquad \fd_DB=\infty,
 \qquad \Tor_1^B(E_0,S)\ne0.
\]
Thus $S$ is Gorenstein flat over $D$ but not over $B$.
\end{enumerate}
\end{maintheorem}

The construction of $D$, its dimension, and its polynomial-fiber
obstruction are due to \cite{TVN}. The additional assertions here are
its Gorenstein Pr\"ufer property, the uniform complete resolutions, and
the specified noninvertibility and Tor obstruction. In particular, the
matrix-factorization construction over a hypersurface is Eisenbud's
\cite{Eisenbud}; our argument chooses matrices of fixed size for a
prescribed generating row and proves that their exactness and dual
exactness survive the ultraproduct and the subsequent quotient. The
proof is independent of any converse characterization using all
overrings of a domain.

%

Section~\ref{sec:transfer} develops the permanence results and the uniform
complete resolutions. Section~\ref{sec:construction} specifies the
one-dimensional domain and transfers those resolutions to it.
Section~\ref{sec:obstruction} identifies the noninvertible ideal and the
Tor obstruction, and compares the example with simpler curves and with
a nonperiodic Gorenstein Pr\"ufer example.
\section{Divisoriality and uniform complete resolutions}\label{sec:transfer}
\label{sec:prelim}

This section separates two permanence questions. Divisoriality can be
transferred by a finite separation argument, whereas coherence needs a
uniform bound on relations. Matrix factorizations then strengthen those
ideal-theoretic conclusions to complete resolutions, and flat divided
primes permit their reduction.

\subsection{Conventions, divisoriality, and localization}

All rings are commutative with identity, homomorphisms preserve the
identity, and modules are unital. Domains have $1\ne0$. The normalization
$\overline A$ of a domain $A$ is its integral closure in $K=\Frac A$;
``normal'' and ``integrally closed'' have this meaning. An overring is
a ring between $A$ and $K$. A fractional ideal is a nonzero submodule
$I\subseteq K$ with $dI\subseteq A$ for some $0\ne d\in A$.
For nonzero fractional ideals $I,J$ of a domain $A$ with quotient
field $K$, set $(I:J)=\{x\in K:xJ\subseteq I\}$. In particular, put $I^{-1}=(A:I)$ and $I_v=(A:(A:I))$.
\begin{definition}\label{def:fgv}
An ideal is \emph{divisorial} if $I=I_v$. We call $A$ an
\emph{FGV-domain} if every nonzero finitely generated integral ideal
is divisorial. Scaling shows that the same then holds for nonzero
finitely generated fractional ideals. This condition does not require
all ideals of $A$ to be divisorial.
A ring is \emph{coherent} if every finitely generated ideal is finitely
presented, and a domain is \emph{Pr\"ufer} if every nonzero finitely
generated ideal is invertible.
\end{definition}

\begin{definition}\label{def:normalization-gf}
A module is Gorenstein flat if it is a cycle in an
exact, doubly infinite complex of flat modules which remains exact
after tensoring with every injective module. Its Gorenstein flat
dimension is the least length of a resolution by such modules, or
infinity if there is no finite resolution; see \cite{CFH}. We put $\Gfd_A0=0$ and write $\wGgldim(A)=\sup_M\Gfd_AM$.
A domain is \emph{Gorenstein Pr\"ufer}, or $G$-Pr\"ufer, if it is
coherent and $\wGgldim(A)\le1$.
\end{definition}

We recall the characterization used throughout.

\begin{theorem}[Qiao--Wang; Hu--Lim--Zhou]
\label{thm:criterion}
For a domain $A$, the following are equivalent:
\begin{enumerate}
\item $A$ is $G$-Pr\"ufer.
\item $A$ is coherent and FGV.
\item Every torsion-free $A$-module is Gorenstein flat.
\end{enumerate}
\end{theorem}

The equivalence of the first two conditions is
\cite[Theorem~4.2]{QW}; the torsion-free characterization is
\cite[Corollary~2.4]{HLZ}. In particular, every overring of a
$G$-Pr\"ufer domain is Gorenstein flat over the base domain.

The next elementary observation will be used for localization and
for comparison with the classical theory.

\begin{lemma}\label{lem:inverse}
Let $A$ be a coherent domain and let $I$ be a nonzero finitely
generated fractional ideal. Then $I^{-1}$ is a finitely generated
fractional ideal. If $S\subseteq A\setminus\{0\}$ is multiplicatively
closed, then
\begin{equation}\label{eq:localv}
 (I^{-1})_S=(A_S:I_S),\qquad (I_v)_S=(I_S)_v,
\end{equation}
where the final $v$-operation is taken over $A_S$. Consequently,
every localization of a coherent FGV-domain is a coherent FGV-domain.
\end{lemma}

\begin{proof}
After scaling, assume $I=(a_1,\ldots,a_r)\subseteq A$, with each
$a_i\ne0$. Then $I^{-1}=\bigcap_i a_i^{-1}A$. A common nonzero
multiple of this intersection is an intersection of finitely many
principal integral ideals, which is finitely generated by coherence.
This proves the first assertion.

For the first equality in \eqref{eq:localv} follows from \cite[Proposition 1.5.13(3)]{WangKim2024}.
Apply it again to the finitely generated fractional
ideal $I^{-1}$ to obtain the second equality. Coherence localizes,
and every finitely generated ideal of $A_S$ is the extension of a
finitely generated ideal of $A$. The last assertion follows.
\end{proof}

We will also use two standard consequences. An integrally closed
$G$-Pr\"ufer domain is Pr\"ufer \cite[p.~1450]{HWX}, and every
one-dimensional Noetherian Gorenstein local domain is $G$-Pr\"ufer;
see \cite{CFH,HLZ}. In the latter case, the double-annihilator property
in the Artinian Gorenstein quotients $A/aA$ shows that every ideal
of $A$ is divisorial.

The following transfer arguments use finite lists of coefficients.
A uniform bound on relation generators will ensure coherence of the
ultraproduct.

\subsection{Ultraproducts and uniform coherence}

We first recall the ultraproduct terminology and establish the
coordinatewise ideal-membership criterion used in the uniform
coherence arguments.

Let $\Lambda$ be an infinite set. An \emph{ultrafilter} $\U$ on
$\Lambda$ is a family of subsets of $\Lambda$ containing $\Lambda$
but not $\varnothing$, closed under finite intersections and passage
to supersets, and satisfying the following dichotomy: for each
$S\subseteq\Lambda$, exactly one of $S$ and $\Lambda\setminus S$
belongs to $\U$. The ultrafilter is \emph{nonprincipal} if it contains
no finite subset of $\Lambda$, or equivalently, if it contains every
cofinite subset of $\Lambda$.

A property holds for \emph{$\U$-almost all} $\lambda\in\Lambda$ if
the set of indices for which it holds belongs to $\U$. In particular,
for a nonprincipal ultrafilter on $\{2,3,\ldots\}$, every property
that holds for all sufficiently large indices holds for $\U$-almost
all indices. The elementary facts below apply to arbitrary
ultrafilters; nonprincipality will be used in the subsequent
construction.

For a family $(A_\lambda)_{\lambda\in\Lambda}$ of commutative rings
with identity, define
\[
 \mathcal Z_{\U}
 =
 \left\{(a_\lambda)\in\prod_{\lambda\in\Lambda}A_\lambda:
   \{\lambda\in\Lambda:a_\lambda=0\}\in\U
 \right\}.
\]
This is an ideal of the direct product. The \emph{ultraproduct} is
the quotient ring
\[
 A=\prod_{\U}A_\lambda
   =\left(\prod_{\lambda\in\Lambda}A_\lambda\right)/\mathcal Z_{\U}.
\]
We denote the class of $(a_\lambda)$ by $[a_\lambda]$. Thus
$[a_\lambda]=[b_\lambda]$ precisely when $a_\lambda=b_\lambda$ for
$\U$-almost all $\lambda$. Addition and multiplication are induced
by the coordinatewise operations.

If every $A_\lambda$ is an integral domain, then $A$ is an integral
domain. Indeed, if $[a_\lambda][b_\lambda]=0$, then
$\{\lambda:a_\lambda b_\lambda=0\}\in\U$. Since each coordinate ring
is a domain, this set is the union of
$\{\lambda:a_\lambda=0\}$ and $\{\lambda:b_\lambda=0\}$.
The ultrafilter property places at least one of these two sets in
$\U$, so one of the two classes is zero. Moreover, $1\ne0$ in $A$.

The next lemma shows that membership in an ideal with a fixed finite
list of generators can be checked coordinatewise.

\begin{lemma}\label{lem:ultraproduct-ideal-membership}
Let $\U$ be an ultrafilter on a set $\Lambda$, let
$(A_\lambda)_{\lambda\in\Lambda}$ be a family of commutative rings
with identity, and put $A=\prod_{\U}A_\lambda$.
Fix a positive integer $r$ and elements
$a_i=[a_{i,\lambda}]\in A$ for $1\le i\le r$. Set
$I=(a_1,\ldots,a_r)A$ and
$I_\lambda=(a_{1,\lambda},\ldots,a_{r,\lambda})A_\lambda$.
For every $z=[z_\lambda]\in A$, one has
\begin{equation}\label{eq:membership}
 z\in I
 \quad\text{if and only if}\quad
 z_\lambda\in I_\lambda
 \quad\text{for $\U$-almost all $\lambda\in\Lambda$}.
\end{equation}
\end{lemma}

\begin{proof}
Suppose $z=\sum_{i=1}^{r}a_i b_i$ for elements
$b_i=[b_{i,\lambda}]\in A$. Equality in the ultraproduct means that
$z_\lambda=\sum_{i=1}^{r}a_{i,\lambda}b_{i,\lambda}$ for
$\U$-almost all $\lambda$. Hence $z_\lambda\in I_\lambda$ on a
member of $\U$.

Conversely, suppose $S=\{\lambda:z_\lambda\in I_\lambda\}$ belongs
to $\U$. For each $\lambda\in S$, choose coefficients
$b_{1,\lambda},\ldots,b_{r,\lambda}\in A_\lambda$ such that
$z_\lambda=\sum_{i=1}^{r}a_{i,\lambda}b_{i,\lambda}$, and set all
these coefficients equal to zero when $\lambda\notin S$.
Then $b_i=[b_{i,\lambda}]\in A$ satisfies
$z=\sum_{i=1}^{r}a_i b_i$, so $z\in I$.
\end{proof}

The fixed number of generators allows the coordinatewise
coefficients to be assembled into finitely many elements of the
ultraproduct. In the coherence arguments below, an analogous
uniform bound on the generators of relation modules will be needed.

Divisoriality has the separation form needed for this transfer.

\begin{proposition}\label{prop:ultra-fgv}
Let $\Lambda$ be a nonempty set, let $\mathcal U$ be an ultrafilter on
$\Lambda$, and let each $A_\lambda$ be an FGV-domain. Then
$\prod_{\mathcal U}A_\lambda$ is an FGV-domain.
\end{proposition}

\begin{proof}
Let $A=\prod_{\U}A_\lambda$, with each $A_\lambda$ FGV, and take a
nonzero finitely generated ideal $I=(a_1,\ldots,a_r)$ of $A$.
For $z\in A\setminus I$, put
$I_\lambda=(a_{1,\lambda},\ldots,a_{r,\lambda})A_\lambda$.
By \eqref{eq:membership}, there is a member $E$ of $\U$ on which
$I_\lambda\ne0$ and $z_\lambda\notin I_\lambda$.

For each $\lambda\in E$, divisoriality supplies
$q_\lambda\in\Frac(A_\lambda)$ such that
$q_\lambda I_\lambda\subseteq A_\lambda$ but
$q_\lambda z_\lambda\notin A_\lambda$. Write
$q_\lambda=b_\lambda/c_\lambda$, with $c_\lambda\ne0$, and put
$b_\lambda=0$, $c_\lambda=1$ outside $E$. Then
$q=[b_\lambda]/[c_\lambda]\in\Frac(A)$ satisfies $qI\subseteq A$.
If $qz\in A$, representatives for this product would give
$q_\lambda z_\lambda\in A_\lambda$ almost everywhere, a
contradiction. Hence $z\notin I_v$.

Since $1\in I^{-1}$, one has $I_v\subseteq A$. The separation
argument proves $I_v\subseteq I$, and the reverse inclusion is
automatic.
\end{proof}

Coherence requires a separate uniform bound. For a finite row
$\boldsymbol a=(a_1,\ldots,a_r)$ over a ring $B$, put
$\Rel_B(\boldsymbol a)=\ker(B^r\xrightarrow{\boldsymbol a}B)$.

\begin{lemma}\label{lem:uniform}
Let $\Lambda$ be a nonempty set, let $\mathcal U$ be an ultrafilter on
$\Lambda$, and let $(A_\lambda)_{\lambda\in\Lambda}$ be domains.
Suppose that for
each $r\geq1$ there is an integer $h(r)$ such that every relation
module on $r$ elements of every $A_\lambda$ can be generated by at
most $h(r)$ elements. Then $\prod_{\U}A_\lambda$ is coherent.
If all $A_\lambda$ are FGV-domains, the ultraproduct is $G$-Pr\"ufer.
\end{lemma}

\begin{proof}
For a row over the ultraproduct, choose $h(r)$ generators of the
relation module in each coordinate, padding with zeros. Their
classes are relations over the ultraproduct. An arbitrary relation
holds coordinatewise on a member of $\U$ and there is a linear
combination of the chosen generators on that set. Extend its
coefficients by zero elsewhere. This proves that the chosen
$h(r)$ classes generate the relation module. The finite-row
criterion proves coherence. The final assertion follows from
Proposition~\ref{prop:ultra-fgv} and Theorem~\ref{thm:criterion}.
\end{proof}

\subsection{Divided-prime quotients}

The required quotient operation is much more restrictive than
passage to an arbitrary prime quotient.

\begin{definition}\label{def:divided}
A prime ideal $P$ of a domain $A$ will be called \emph{divided} if
\begin{equation}\label{eq:divided}
 aP=P\qquad\text{for every }a\in A\setminus P.
\end{equation}
Equivalently, $P$ is comparable by inclusion with every principal
ideal of $A$. We use the explicit condition \eqref{eq:divided}
throughout.
\end{definition}

For the equivalence, if $a\notin P$ and $P\subseteq aA$, writing
$p=ab$ and using primality gives $b\in P$. The converse follows
immediately from \eqref{eq:divided}.

The divided condition identifies the coefficient domain inside its
localization and ensures that a nonzero lifted ideal contains the prime.

\begin{lemma}\label{lem:pullback}
Let $P$ be a divided prime of a domain $A$, put $B=A/P$, and let
$\pi:A_P\to\Frac(B)$ be the residue map. Then
\begin{equation}\label{eq:pullback}
 PA_P=P\quad\text{inside }\Frac(A),
 \qquad \pi^{-1}(B)=A\quad\text{inside }A_P.
\end{equation}
Moreover, for any nonzero finitely generated ideal
$J=(\bar a_1,\ldots,\bar a_r)$ of $B$, the lifted ideal
$I=(a_1,\ldots,a_r)A$ contains $P$ and $J=I/P$.
\end{lemma}

\begin{proof}
For $p\in P$ and $s\notin P$, the identity $sP=P$ gives
$p/s\in P$, proving $PA_P=P$. Localizing $A/P$ at all its nonzero
elements identifies $A_P/PA_P$ with $\Frac(B)$.
If $x\in A_P$ and $\pi(x)=\bar a\in B$, choose $a\in A$ lifting
$\bar a$. Then $x-a\in PA_P=P\subseteq A$, so $x\in A$.
This proves the preimage identity. Finally, some $a_j\notin P$,
and $P=a_jP\subseteq I$, yielding the last assertion.
\end{proof}

The preimage identity transfers separation by fractional multipliers
through the quotient map.

\begin{theorem}\label{thm:quotient}
Let $A$ be a domain and let $P$ be a divided prime.
\begin{enumerate}
\item If $A$ is FGV, then $A/P$ is FGV.
\item If $A$ is coherent, then $A/P$ is coherent.
\item If $A$ is $G$-Pr\"ufer, then $A/P$ and every localization
of $A/P$ are $G$-Pr\"ufer.
\end{enumerate}
\end{theorem}

\begin{proof}
Set $B=A/P$. For (i), take a nonzero finitely generated ideal
$J=(\bar a_1,\ldots,\bar a_r)$ of $B$, and lift it to
$I=(a_1,\ldots,a_r)A$. Choose $a_j\notin P$.
By Lemma~\ref{lem:pullback}, $P\subseteq I$ and $J=I/P$.
Given $\bar z\in B\setminus J$, a lift $z\in A$ lies outside $I$.
As $I$ is divisorial, there exists $q\in(A:I)$ with $qz\notin A$.
Since $qa_j\in A$ and $a_j\notin P$, one has $q\in A_P$.
The element $\pi(q)\in\Frac(B)$ maps $J$ into $B$.
If $\pi(q)\bar z$ belonged to $B$, then $qz$ would belong to
$\pi^{-1}(B)=A$, a contradiction. Thus $\bar z\notin J_v$.
As before, $J_v\subseteq B$, so $J$ is divisorial.

For (ii), take a nonzero row $(\bar a_1,\ldots,\bar a_r)$ over
$B$ and a lift with $a_j\notin P$. A lifted relation
$(b_1,\ldots,b_r)$ has error $p=\sum_i a_ib_i\in P$.
Write $p=a_jc$ with $c\in P$. Replacing $b_j$ by $b_j-c$ gives
an exact relation over $A$ with the same reduction. Therefore
\[
 \Rel_A(a_1,\ldots,a_r)\longrightarrow
 \Rel_B(\bar a_1,\ldots,\bar a_r)
\]
is surjective. Coherence of $A$ makes the target finitely generated.
The zero row has relation module $B^r$, so the finite-row criterion
proves coherence of $B$. Notice that no finite generation of $P$
has been assumed.

The relation map also preserves any uniform bound for nonzero rows;
for a zero row of length $r$, the relation module is $B^r$.
Part (iii) follows from (i), (ii), Theorem~\ref{thm:criterion},
and Lemma~\ref{lem:inverse}.
\end{proof}

The coherence argument in part (ii) abstracts the relation-lifting
argument in \cite[Proposition~2.3]{TVN}. The separation argument in (i) additionally transfers FGV.

\subsection{Uniform matrix factorizations}\label{sec:periodic}

A complex is \emph{two-periodic} if its modules and differentials repeat
after a shift by two. We assert the existence of such a complex, not
minimality or least period two. A totally acyclic complex of finitely
generated free modules is an exact complex of finitely
generated free modules whose ring dual is exact. Each cycle is both
Gorenstein projective and Gorenstein flat. Indeed, for an injective
module $E$, finite free duality identifies its tensor product with the
complex with the result of applying $\Hom(-,E)$ to the exact dual.
For an arbitrary projective target, Hom-exactness follows by taking
sums of the finite free dual complex and then direct summands.

The matrix factorizations below are chosen with a size depending only
on the length of the prescribed generating row. This uniformity permits
the exact complexes and their duals to pass to an ultraproduct.

\begin{theorem}\label{thm:uniform-periodic}
Let $\Lambda$ be a nonempty set and let
$(B_\lambda)_{\lambda\in\Lambda}$ be domains of the form
$B_\lambda=k_\lambda[[X,Y]]/(f_\lambda)$, where $k_\lambda$ is a field
and $f_\lambda$ is a nonzero nonunit. Let $\U$ be an ultrafilter and put
$A=\prod_\U B_\lambda$. Every nonzero ideal $I$ of $A$ generated by
$r\ge1$ elements admits a two-periodic totally acyclic complex with
all terms $A^r$ and with $I$ as a cycle. In particular, $A$ is coherent
and $G$-Pr\"ufer.
\end{theorem}
\begin{proof}
\step{Step 1: matrix factorizations of the prescribed size.}Write $I=(a_1,\ldots,a_r)A$ and choose coordinate rows
$(a_{1,\lambda},\ldots,a_{r,\lambda})$ with image $I_\lambda$.
Let $S_\lambda=k_\lambda[[X,Y]]$. If $I_\lambda\ne0$, then
$I_\lambda$ is a maximal Cohen--Macaulay module of dimension one over
$B_\lambda$. Viewed over $S_\lambda$, it has depth one: it is nonzero torsion-free over the one-dimensional
local domain $B_\lambda$. The regular local ring $S_\lambda$ has global
dimension two, so the Auslander--Buchsbaum formula gives projective
dimension one
\cite[Tag~090U]{Stacks}. Hence the kernel of the prescribed
surjection $S_\lambda^r\to I_\lambda$ is a free module. Its rank is
$r$, since $f_\lambda I_\lambda=0$. Choosing a basis gives
\[
 0\longrightarrow S_\lambda^r
 \xrightarrow{\Phi_\lambda}S_\lambda^r
 \longrightarrow I_\lambda\longrightarrow0.
\]
As $f_\lambda S_\lambda^r\subseteq\im\Phi_\lambda$, there is
$\Psi_\lambda$ with
\[
 \Phi_\lambda\Psi_\lambda
 =\Psi_\lambda\Phi_\lambda=f_\lambda I_r.
\]
The second equality follows from injectivity of $\Phi_\lambda$.
For a zero ideal use $\Phi_\lambda=I_r$ and
$\Psi_\lambda=f_\lambda I_r$.

\step{Step 2: exactness and dual exactness.}Reducing modulo $f_\lambda$ gives matrices
$\varphi_\lambda,\psi_\lambda$ over $B_\lambda$ and an exact
alternating complex with cokernel $I_\lambda$. For example,
$\Phi_\lambda v=f_\lambda w$ implies
$\Phi_\lambda(v-\Psi_\lambda w)=0$, proving
$\ker\varphi_\lambda=\im\psi_\lambda$; the other equality is
similar. Both lifted matrices are injective, and the same argument
applies to their transposes. This is the matrix-factorization
construction of \cite{Eisenbud}.

\step{Step 3: ultraproducts and coherence.}Take the ultraproduct of the entries of the fixed-size matrices.
If a vector lies in a kernel over $A$, its representative lies in the
corresponding coordinate kernel on a set in $\U$. Choose preimages
coordinatewise on that set, and put zero entries elsewhere. This
proves exactness of the alternating complex over $A$. The same
argument proves exactness of its dual and of its augmentation onto
$I$. Thus $\operatorname{coker}\varphi\cong I$, and the alternating
complex has a cycle isomorphic to $I$.

In particular, the kernel of $A^r\to I$ is generated by at most $r$
columns. The zero row also has a finitely generated kernel.
Thus $A$ is coherent. Every finitely generated ideal is Gorenstein
flat; equivalently, by the criterion in \cite[Corollary~2.4]{HLZ},
$A$ is $G$-Pr\"ufer.
\end{proof}

The period in this statement divides two; neither minimality of the
complex nor least period two is asserted. Its free ranks are tied to
the prescribed generating row rather than to a minimal resolution.

\subsection{Preserving complete complexes through a quotient}

A flat prime permits reduction of both a complete complex and its
dual, while the divided-prime condition identifies the reduced ideal.

\begin{lemma}\label{lem:flat-quotient-complex}
Let $A$ be a domain and $P$ a flat prime ideal. Put $B=A/P$.
If $\boldsymbol F$ is a totally acyclic complex of finitely generated
free $A$-modules, then $B\otimes_A\boldsymbol F$ is a totally acyclic
complex of finitely generated free $B$-modules. Its cycles are the
modules $B\otimes_A Z$ for the corresponding cycles $Z$ of
$\boldsymbol F$. Periodicity and the ranks of the terms are preserved.
\end{lemma}
\begin{proof}
The case $P=0$ is immediate. For a torsion-free $A$-module $M$, the
multiplication map $P\otimes_A M\to M$ is injective. Indeed,
$M\hookrightarrow K\otimes_A M$, with $K=\Frac(A)$, remains injective
after tensoring with the flat module $P$, and
$P\otimes_A K\cong K$. Consequently the flat resolution
$0\to P\to A\to B\to0$ gives
$\Tor_i^A(B,M)=0$ for every $i>0$.

All cycles of $\boldsymbol F$ and $\boldsymbol F^*$ are torsion-free,
since they are submodules of free modules. Tensor their short exact
cycle sequences with $B$. The Tor vanishing gives exactness and the
asserted description of cycles. Degreewise finite free duality gives
\[
 \Hom_B(B\otimes_A\boldsymbol F,B)
 \cong B\otimes_A\Hom_A(\boldsymbol F,A),
\]
which is exact by the same argument. Periodicity and ranks are
unchanged by this scalar extension.
\end{proof}

Dividedness supplies the additional ideal identity needed to identify
the reduced cycle with an arbitrary ideal of the quotient.

\begin{corollary}\label{cor:periodic-divided}
Let $A$ be a domain, let $P$ be a flat divided prime, and fix an integer
$r\ge1$. Suppose every nonzero $r$-generated ideal of $A$ has a two-periodic totally acyclic complex
with terms $A^r$. Then every nonzero $r$-generated ideal of $A/P$, and
of every localization of $A/P$, has a two-periodic totally acyclic
complex with free terms of rank $r$.
\end{corollary}
\begin{proof}
Lift a nonzero ideal $J=(\bar a_1,\ldots,\bar a_r)$ of $A/P$ to
$I=(a_1,\ldots,a_r)A$, with some $a_j\notin P$. Dividedness gives
$P=a_jP\subseteq PI\subseteq P$, hence $PI=P$. Therefore
$B\otimes_A I=I/P=J$. Apply Lemma~\ref{lem:flat-quotient-complex}.
For a localization, clear a common denominator from the prescribed
$r$ generators and localize the complex; multiplying by that
invertible denominator preserves the isomorphism class of the ideal.
\end{proof}

\section{A one-dimensional local domain with periodic ideal resolutions}\label{sec:construction}

We now apply the permanence results to a specified order-cut
construction. The prime ideals, dimension, and polynomial fiber are
recalled with their original attribution. The additional work here
proves flatness of the divided prime and transfers the uniform
complete resolutions.

\subsection{The construction and the cited input}

We now specialize the preceding ultraproduct constructions to the
plane-curve domains used in the order-cut construction of
\cite[Section~2]{TVN}. Fix a nonprincipal ultrafilter $\U$ on
$\{2,3,\ldots\}$ and set
\begin{equation}\label{eq:factors}
 A_n=\mathbb Q[[t^n,t^{n+1}]],
 \qquad
 A=\prod_{\U}A_n.
\end{equation}
Throughout this construction, statements holding for
$\U$-almost all $n$ are understood with respect to this fixed
ultrafilter.

Let $\nu_n$ be the $t$-adic order, with $\nu_n(0)=\infty$, and define
\begin{align}
 P&=\{[a_n]:(\forall c\in\mathbb Z_{\ge1})\,
           \{n:\nu_n(a_n)>cn^2\}\in\U\},\label{eq:P}\\
 Q&=\{[a_n]:(\exists\epsilon>0)\,
           \{n:\nu_n(a_n)>\epsilon n^2\}\in\U\},\label{eq:Q}\\
 D&=(A/P)_{Q/P},\qquad \m=(Q/P)D.\label{eq:D}
\end{align}
The bounds associated with different constants may hold on different
members of $\U$.

For the obstruction, set
\[
 \alpha=\sqrt2-1,\qquad d_n=\lfloor\alpha n\rfloor,\qquad
 u=[t^{n^2}],\qquad v=[t^{n^2+d_n}].
\]
Use $\bar u,\bar v$ for their images in $D$ and set $z=\bar v/\bar u$.
The following inputs, including this specified fraction, are due to
Tran, Vo, and Nguyen.

\begin{theorem}[Tran--Vo--Nguyen]\label{thm:TVN}
Fix a nonprincipal ultrafilter $\mathcal U$ on $\{2,3,\ldots\}$ and
form $A$, $P$, $Q$, $D$, $\mathfrak m$, $\bar u$, $\bar v$, and $z$
by \eqref{eq:factors}--\eqref{eq:D} and the displayed formulas above.
Then $P\subsetneq Q$ are prime and
$aP=P$ for every $a\in A\setminus P$. The domain $D$ is local of
dimension one, and $0\ne\bar u,\bar v\in\m$. Moreover,
\begin{equation}\label{eq:polynomial-fiber}
 D[z]/\m D[z]\cong (D/\m)[Z],
\end{equation}
where $Z$ maps to the residue of $z$.
\end{theorem}

The prime and dimension assertions are
\cite[Proposition~2.3]{TVN}; the polynomial-fiber calculation is
\cite[Lemma~3.1]{TVN}. The numerical-semigroup conductor inclusion
$t^{n^2}\mathbb Q[[t]]\subseteq A_n$ used below is
\cite[Lemma~2.1]{TVN}.

\subsection{Flatness of the cut and complete resolutions over the quotient}

The order cut defining $P$ has an additional property that permits
complete resolutions to pass to $D$. The proof uses finitely many
coordinate inequalities at a time; it does not take a countable
intersection of members of the ultrafilter.

\begin{lemma}\label{lem:cut-flat}
Let $A=\prod_\mathcal U\mathbb Q[[t^n,t^{n+1}]]$ for a nonprincipal
ultrafilter $\mathcal U$ on $\{2,3,\ldots\}$, and let $P$ be given by
\eqref{eq:P}. Then $P$ is a flat idempotent ideal of $A$.
More precisely, every finite subset of $P$ is contained in a
principal ideal $cA$ with $c\in P$, and may be factored with its
quotients also in $P$.
\end{lemma}
\begin{proof}
Take $p_1,\ldots,p_r\in P$, not all zero, with representatives
$p_{i,n}$. On a member of $\U$, their minimum order
$s_n=\min_i\nu_n(p_{i,n})$ is finite and greater than $4n^2$.
There put $c_n=t^{\lfloor s_n/2\rfloor}$, and put $c_n=0$
outside that set. Both $c_n$ and every $p_{i,n}/c_n$ on this set
have order at least $n^2$, so they belong to $A_n$ by its conductor
inclusion. Define the quotient representatives to be zero elsewhere.

Since $s_n/n^2$ exceeds every fixed bound almost everywhere, both
$c=[c_n]$ and $q_i=[p_{i,n}/c_n]$ belong to $P$, and $p_i=cq_i$.
This proves the factorization assertion and $P^2=P$.
The principal subideals $cA\subseteq P$ with $c\ne0$ form a directed
family whose union is $P$. Each is a free rank-one $A$-module, so
$P$ is flat as their filtered colimit.
\end{proof}

We now obtain the Gorenstein property and the uniform resolutions.

\begin{theorem}\label{thm:main}
Let $\mathcal U$ be a nonprincipal ultrafilter on $\{2,3,\ldots\}$ and
let $D=(A/P)_{Q/P}$ be the domain defined in
\eqref{eq:factors}--\eqref{eq:D}. Then $D$ is a one-dimensional local
$G$-Pr\"ufer domain. For every integer $r\ge1$, every nonzero ideal of $D$ generated by $r$
elements admits a two-periodic complete free resolution with terms
$D^r$. Its integral closure $B=\overline D$ is not Pr\"ufer.
\end{theorem}
\begin{proof}
\step{Step 1: the Gorenstein property and its periodic refinement.}Each $A_n\cong\mathbb Q[[X,Y]]/(Y^n-X^{n+1})$ is a plane-curve
domain. Theorem~\ref{thm:uniform-periodic} makes $A$ coherent and
$G$-Pr\"ufer and supplies the uniform periodic complexes.
Theorem~\ref{thm:TVN} makes $P$ divided, and
Lemma~\ref{lem:cut-flat} makes it flat. Thus
Corollary~\ref{cor:periodic-divided} transfers the complexes to
$D$. Theorem~\ref{thm:quotient} and localization also give that
$D$ is coherent and FGV, hence $G$-Pr\"ufer. Its dimension and
locality are in Theorem~\ref{thm:TVN}.

\step{Step 2: failure of Pr\"ufer normalization.}The primes of the polynomial
fiber \eqref{eq:polynomial-fiber} give
\[
 0\subsetneq\m D[z]\subsetneq\m D[z]+zD[z].
\]
Thus $\dim D[z]\ge2$. If $B$ were Pr\"ufer, its integral extension
over $D$ would give $\dim B=1$. Every overring of a one-dimensional
Pr\"ufer domain has dimension at most one. But $B[z]$ is integral
over $D[z]$, so dimension preservation would give
$2\le\dim D[z]=\dim B[z]\le1$, a contradiction. These dimension
arguments are the ones in \cite[Lemma~3.2 and the proof of
Theorem~1.1]{TVN}.
\end{proof}

The ideal-theoretic proof of the $G$-Pr\"ufer property can also be
read directly from Proposition~\ref{prop:ultra-fgv},
Lemma~\ref{lem:uniform}, and Theorem~\ref{thm:quotient}, using the
uniform relation bound of \cite[Lemma~2.2]{TVN}. Theorem~\ref{thm:main}
additionally identifies a periodic complete resolution for every
prescribed finite generating row in the final domain.

\section{Normalization and the limits of Gorenstein flatness}
\label{sec:obstruction}

The polynomial fiber supplies an ideal-theoretic obstruction that
persists in every integral overring. Over the normalization, a two-chart
sequence converts this obstruction into nonvanishing Tor. We conclude
by separating the resulting example from familiar nonnormal curves and
by showing that two-periodicity is not automatic for Gorenstein
Pr\"ufer domains.

\subsection{A noninvertible ideal and the two-chart Tor test}

The next argument uses the polynomial-fiber identity itself and does
not impose a dimension bound.

\begin{lemma}\label{lem:polynomial-fiber-ideal}
Let $(R,\mathfrak m,k)$ be a local domain, let $0\ne a,b\in R$, and
put $z=b/a$. Suppose the natural map
$k[Z]\to R[z]/\mathfrak mR[z]$ is an isomorphism.
Then $(a,b)T$ is not invertible for any integral overring $T$ of $R$.
\end{lemma}
\begin{proof}
\step{Step 1: comparable primes over one prime of the integral overring.}The ideals
$\mathfrak q_0=\mathfrak mR[z]\subsetneq
\mathfrak q_1=\mathfrak mR[z]+zR[z]$
are prime and both contract to $\mathfrak m$ in $R$.
The extension $R[z]\subseteq T[z]$ is integral. Lying-over and
going-up give primes $\mathfrak Q_0\subsetneq\mathfrak Q_1$ of
$T[z]$ above them. Their contractions to $T$ coincide, say with
$\mathfrak p$, by incomparability for the integral extension
$R\subseteq T$; see \cite[Section~10.36]{Stacks}.

\step{Step 2: invertibility contradicts the prime chain.}Suppose $(a,b)T$ is invertible. Over the local domain $T_{\mathfrak p}$
it is generated by either $a$ or $b$. If $b$ generates it, then
$z^{-1}=a/b\in T_{\mathfrak p}$, so $z$ is a unit in
$T_{\mathfrak p}[z]$. This contradicts the survival of
$\mathfrak Q_1$, which contains $z$, after localizing at
$T\setminus\mathfrak p$.
If $a$ generates the ideal, then $z\in T_{\mathfrak p}$, so
$(T\setminus\mathfrak p)^{-1}T[z]=T_{\mathfrak p}$.
Both $\mathfrak Q_0$ and $\mathfrak Q_1$ survive this localization
and contract to the maximal ideal $\mathfrak pT_{\mathfrak p}$.
They must therefore become the same prime, contradicting the
prime correspondence for localization. This excludes invertibility.
\end{proof}

For an integrally closed domain, noninvertibility is detected by a
first Tor group for a two-chart overring.

\begin{lemma}\label{lem:normal-tor-obstruction}
Let $B$ be an integrally closed domain, let $0\ne a,b\in B$, and
suppose $I=(a,b)B$ is not invertible. Put $z=b/a$ and
$S=B[z,z^{-1}]$. For every injective cogenerator $E_0$ of $B$-modules (that is,
$\Hom_B(N,E_0)\ne0$ for every nonzero $B$-module $N$),
\[
 \Tor_1^B(E_0,S)\ne0.
\]
In particular, $S$ is not Gorenstein flat over $B$.
\end{lemma}
\begin{proof}
Normality gives $B[z]\cap B[z^{-1}]=B$, and their sum is $S$.
Thus \cite[Tag~052J]{Stacks} gives
\[
 0\longrightarrow B\xrightarrow{c\mapsto(c,-c)}
 B[z]\oplus B[z^{-1}]\longrightarrow S\longrightarrow0.
\]
Suppose $\Tor_1^B(E_0,S)=0$. Then the induced map from $E_0$ to
$(E_0\otimes_B B[z])\oplus(E_0\otimes_B B[z^{-1}])$ is injective.
Since $II^{-1}\ne B$, the cogenerator property supplies
$0\ne e\in E_0$ with $(II^{-1})e=0$.

Every relation $ra+sb=0$ has the form
$(r,s)=(bc,-ac)$ for $c\in I^{-1}$. Both coordinates annihilate
$e$, so the assignments
$a\mapsto e,b\mapsto0$ and $a\mapsto0,b\mapsto e$
define maps $I\to E_0$. Extend them to $B$ by injectivity.
There are $e_1,e_2\in E_0$ with
\[
 ae_1=e,\quad be_1=0,\qquad ae_2=0,\quad be_2=e.
\]
Since $b=az$, the image of $e$ in $E_0\otimes_B B[z]$ is
$be_2\otimes1=ae_2\otimes z=0$.
Using $a=bz^{-1}$ and $e_1$ proves that its image in the other
summand is also zero, a contradiction. Positive Tor against
injectives vanishes for a Gorenstein flat module, proving the last
assertion.
\end{proof}

\subsection{The specified overring and the homological invariants}

The preceding lemmas identify both the ideal and the overring in
Theorem~A, while keeping the two base rings distinct.

\begin{theorem}\label{thm:specified-overring}
Let $\mathcal U$ be a nonprincipal ultrafilter on $\{2,3,\ldots\}$,
let $D=(A/P)_{Q/P}$ be the domain in
\eqref{eq:factors}--\eqref{eq:D}, and put $B=\overline D$.
Use $u=[t^{n^2}]$, $v=[t^{n^2+\lfloor(\sqrt2-1)n\rfloor}]$,
and $z=\bar v/\bar u$ as in Section~\ref{sec:construction}. Then
$(\bar u,\bar v)T$ is not invertible in any integral overring $T$
of $D$. For the overring $S=B[z,z^{-1}]$, one has
\[
 \Gfd_D B=\Gfd_D S=0,\qquad \fd_D B=\infty,
 \qquad \Tor_1^B(E_0,S)\ne0
\]
for every injective cogenerator $E_0$ of $B$-modules. Hence $S$ is
Gorenstein flat over $D$ but not over $B$; $B$ is not a $G$-Pr\"ufer domain.
\end{theorem}
\begin{proof}
Apply Lemma~\ref{lem:polynomial-fiber-ideal} to
\eqref{eq:polynomial-fiber}, with $a=\bar u,b=\bar v$.
The nonvanishing Tor assertion follows from
Lemma~\ref{lem:normal-tor-obstruction} over $B$.
Both $B$ and $S$ are torsion-free $D$-modules, so their Gorenstein
flatness follows from Theorem~\ref{thm:criterion}.

If $B$ were $G$-Pr\"ufer, its torsion-free module $S$ would be
Gorenstein flat, contrary to the nonvanishing just proved. Thus $B$
is not $G$-Pr\"ufer. In particular $B\ne D$, since $D$ is $G$-Pr\"ufer.
This proper integral overring is not flat over $D$. Otherwise it
would be faithfully flat, by lying-over. For $a,b\in D$, $a\ne0$,
and $b/a\in B$, faithful flatness would give
$b\in aB\cap D=aD$, forcing $b/a\in D$. This would make $B=D$,
contrary to its non-Pr\"ufer property. Since finite flat dimension
agrees with Gorenstein flat dimension \cite[Theorem~2.2]{Bennis},
$\Gfd_D B=0$ and nonflatness imply $\fd_D B=\infty$.
\end{proof}

The same example has exact global dimension invariants over its
base domain.

\begin{corollary}\label{cor:dimensions}
Let $D$ be the domain defined in \eqref{eq:factors}--\eqref{eq:D}.
Then $D$ is not Noetherian, and
\[
 \wGgldim(D)=1,\qquad \wgd(D)=\infty.
\]
\end{corollary}
\begin{proof}
If $D$ were Noetherian, the Krull--Akizuki theorem would make its
one-dimensional normalization Dedekind, contrary to
Theorem~\ref{thm:main}; see \cite[Tag~09IG]{Stacks}.
The $G$-Pr\"ufer property gives $\wGgldim(D)\le1$.
For a nonzero nonunit $a\in D$, the nonzero torsion module $D/aD$
cannot be Gorenstein flat, since a Gorenstein flat module over a
domain embeds in a flat, hence torsion-free, module. Thus the weak
Gorenstein global dimension is exactly one. The module $B$ has
infinite flat dimension by Theorem~\ref{thm:specified-overring},
so $\wgd(D)=\infty$.
\end{proof}

\subsection{Examples delimiting normalization and periodicity}

A nonnormal Gorenstein Pr\"ufer domain can have Pr\"ufer normalization;
therefore nonnormality alone does not explain Theorem~\ref{thm:main}.

\begin{example}\label{ex:cusp}
For a field $k$, let $R=k[[t^2,t^3]]$. Then $R$ is a non-Pr\"ufer
Gorenstein Pr\"ufer domain, but its normalization is the DVR $k[[t]]$.
\end{example}
\begin{proof}
The presentation $R\cong k[[u,v]]/(v^2-u^3)$ is a one-dimensional
Noetherian hypersurface, so $R$ is Gorenstein and hence $G$-Pr\"ufer
\cite{CFH,HLZ}. Its quotient field contains $t=t^3/t^2$, which is
integral over $R$ but not in $R$. Hence $R$ is not normal and is not
Pr\"ufer. The ring $k[[t]]=R+Rt$ is integral over $R$, is integrally
closed, and has the same quotient field. It is therefore the
normalization, and is a DVR.
\end{proof}

The periodicity in Theorem~\ref{thm:main} is a stronger feature than
Gorenstein Pr\"uferness. A complete-intersection curve separates them.

\begin{example}\label{ex:nonperiodic-ideal}
Let $k$ be a field of characteristic zero, set
$R=k[[t^4,t^6,t^9]]$, and put $\mathfrak n=(t^4,t^6,t^9)R$.
Then $R$ is Gorenstein Pr\"ufer, but $\mathfrak n$ has no eventually
periodic resolution by finitely generated free $R$-modules.
More precisely, $\beta_i^R(\mathfrak n)=2i+3$ for every $i\ge0$.
\end{example}
\begin{proof}
The assignments $u\mapsto t^4$, $v\mapsto t^6$, and $w\mapsto t^9$
give
\[
 R\cong k[[u,v,w]]/(v^2-u^3,\ w^2-v^3).
\]
Indeed, reduction by the two relations writes every class as
$a(u)+b(u)v+c(u)w+d(u)vw$. The four exponents $0,6,9,15$ have
 distinct residues modulo four, so their substituted series are
independent over $k[[t^4]]$. The first relation defines a domain,
and the second is a nonzero element in that quotient; the displayed
relations are a regular sequence. Thus $R$ is a one-dimensional
Noetherian complete intersection and is $G$-Pr\"ufer \cite{CFH,HLZ}.

Let $\beta_i^R(M)=\dim_k\Tor_i^R(k,M)$ be the $i$th Betti number of a
finitely generated module. The complete-intersection resolution of the
residue field \cite{Tate} gives its Poincar\'e series
\[
 P_k^R(s)=\sum_{i\ge0}\beta_i^R(k)s^i
 =\frac{(1+s)^3}{(1-s^2)^2}
 =\frac{1+s}{(1-s)^2}.
\]
Consequently $\beta_i^R(k)=2i+1$. Since $\mathfrak n$ is the first
syzygy of $k$ in a minimal free resolution,
$\beta_i^R(\mathfrak n)=\beta_{i+1}^R(k)=2i+3$.
An eventually periodic resolution with finitely generated free terms
has bounded ranks and hence bounded Betti numbers, contradicting this
formula. The ideal still has a complete resolution; the conclusion
excludes periodic resolutions of finite rank, not complete resolutions
in general.
\end{proof}

The constructed domain strengthens the known coherent normalization
counterexample in two directions. Divisoriality survives the
ultraproduct and the divided-prime quotient, and flatness of the prime
also preserves the uniform two-periodic complexes. The obstruction
over the normalization is independent of these transfer proofs:
a polynomial fiber forces an extended ideal to remain noninvertible,
and the two-chart sequence detects this by first Tor. No description
of every element of the normalization is required. Finally, the
nonvanishing result asserts $\Gfd_B S>0$ and does not determine
whether this dimension is finite or infinite.

\bigskip

\paragraph*{\bf Competing interests:} The author declares that there are no competing interests.

\medskip

\paragraph*{\bf Availability of data and material:}
No datasets were generated or analysed during the current study.

\medskip

\paragraph*{\bf Declaration of generative AI and AI-assisted technologies in the manuscript preparation process:}  
During the preparation of this manuscript, the author used ChatGPT (OpenAI, GPT-6) for language editing, LaTeX formatting, literature exploration, and assistance in examining mathematical arguments and proofs. All AI-assisted content was independently verified by the author, who takes full responsibility for the accuracy, originality, and integrity of the manuscript.

\end{document}